\documentclass[a4paper, oneside, 12pt]{amsart}
\usepackage[a4paper, total={6in, 8in}]{geometry}
\usepackage[colorlinks=true, citecolor=red]{hyperref}
\usepackage{amsfonts}
\usepackage{amsmath,amssymb,amsfonts,amsthm}
\usepackage{float}

\DeclareMathAlphabet{\mathcal}{OMS}{cmsy}{m}{n}
\usepackage{tikz-cd} 
\usepackage{svg}
\usepackage{float}
\newcommand{\citep}{\cite}
\usepackage[numbers,sort&compress]{natbib}
\newtheorem{theorem}{Theorem}

\newtheorem{lemma}{Lemma}
\newtheorem{remark}{Remark}
\newtheorem{proposition}{Proposition}

\newtheorem{definition}{Definition}

\newcommand{\del}{\partial}

\newcommand{\lcm}{\text{lcm}}
\newcommand{\Arg}{\text{Arg}}
\newcommand{\sgn}{\text{sgn}}
\renewcommand{\Re}{\text{Re}}
\renewcommand{\Im}{\text{Im}}
\renewcommand{\phi}{\varphi}
\newcommand{\eps}{\varepsilon}
\newcommand{\mc}{\mathcal}

\newcommand{\C}{\mathbb{C}}
\newcommand{\Z}{\mathbb{Z}}
\newcommand{\N}{\mathbb{N}}
\newcommand{\D}{\mathbb{D}}
\renewcommand{\H}{\mathbb{H}}
\newcommand{\R}{\mathbb{R}}
\newcommand{\fr}{\frac}

\newcommand{\DD}{\mathbb{D}}

\newcommand{\HH}{\mathbb{H}}
\newcommand{\HHH}{\mathcal{H}}

\newlength\tindent
\title[Inverse problem for $h$-functions of planar Brownian motion]{A method of solution for the inverse problem for $h$-functions of planar Brownian motion} \author{Greg Markowsky,
Clayton McDonald}
\begin{document}
\begin{abstract}
     Let $D$ be a planar domain and $z$ be a point in $D$. The harmonic measure distribution function
     $h^z_D(r)$, with base point $z$, is the harmonic measure with pole at $z$ of
     the parts of the boundary which are within a distance $r$ of $z$. Equivalently, it is the
     probability Brownian motion started from $z$ first strikes the boundary
     within a distance $r$ from $z$. This paper is concerned with
     the following inverse problem: given a suitable function $h$, does there exist a
     domain $D$ such that $h=h_D$?  To answer this, we first extend the concept of
     the $h$-function of a domain to one of a stopping time $\tau$. Then, using
     the conformal invariance of Brownian motion, we solve the inverse problem
     for the $h$-functions of stopping times. In many cases, the stopping time is the exit time of a domain $D$, our technique therefore solves the original
     inverse problem. We will illustrate our methods with a large family of examples.
\end{abstract}
\maketitle
\section{Introduction}
The harmonic measure distribution function, or ``$h$-function'', of a domain $D\subset \C$ 
with base-point $z\in D$ is the probability that planar Brownian motion $Z_t$ started at $z$
first exits $D$ within a distance $r$ from $z$. That is, if we let $\tau_D$ denote the first exit time of Brownian motion, then
$$
	h_D^z(r) := P_z(|Z_{\tau_D} - z| \leq r) = \omega_z(\overline{B_r(z)}\cap \del D).
$$
As far as we are aware, this function was first systematically studied by
Walden and Ward \cite{WW}, who credited Ken Stephensen
\cite{brannan1989research} for the original idea; we note that in those and most subsequent papers on the topic, the $h$-function was defined analytically rather than probabilistically, with an equivalent definition in terms of harmonic measure. Since that time, a large
number of papers have appeared on the topic. Brownian motion is conformally
invariant (equivalently, harmonic measure is as well), so it is evident that a
major tool for attacking this problem is conformal mappings. Many papers have
utilized this approach, such as
\cite{walden2001asymptotic,snipes2005realizing,mahenthiram24new,barton2014new,barton2003conditions,
brannan1989research,mahenthiram2024harmonic,mahenthiram2025computing,snipes2008convergence,
mahenthiram2023computing,green2024towards,
betsakos1998harmonic, betsakos2003distribution}.\\

To give a concrete example, in the upper-half plane $\H:=\{z\in\C:\Im{z}>0\}$, the probability density 
$P_i(Z_{\tau_D}\in ds)$ is well known to be the Poisson kernel on $\H$ (See \cite{greg,chin2020note} for
derivations of this formula). Then for $r\geq 1$ we have
$$
	h_\H^i(r) = \omega_i(B_r(i)\cap\R) = \fr{2}{\pi}\int_0^{\sqrt{r^2-1}} \fr{1}{t^2+1}dt = 
	\fr{2}{\pi}\arctan\sqrt{r^2-1},
$$
and $h_\H^i(r) = 0$ for $r<1$. A diagram of this $h$-function and
its derivation is given below.
\begin{figure}[H]
        \includegraphics[width=150pt, height=100pt]{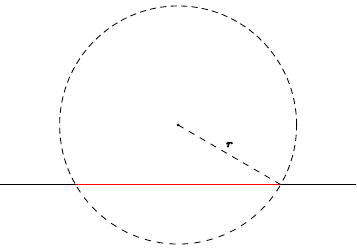}
\hspace{10mm}
        \includegraphics[width=120pt, height=85pt]{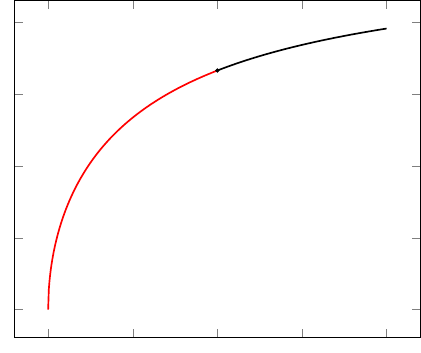}
        \caption{ $h$-function of $\H$ with base point $i$}\label{fig1}
\end{figure}

It is natural to ask which properties of a domain $D$ are reflected in properties of $h_D^z$, and several things are known. For example, it is shown in \cite{MALWsurv} by Buerling's projection theorem that
the $h$-function of a simply-connected domain $D$ must satisfy
$$
	h_D^z(r) \leq 1-\fr{4}{\pi}\arctan\sqrt{\fr{d}{r}},
$$
where $d$ is the minimum distance to the boundary from $z$. Other geometric
aspects of the $h$-function can also be found in \cite{MALWsurv}. Another natural question is uniqueness: it is clear that the domain associated to a given $h$-function is, in general, not unique, because Brownian motion is rotationally
invariant, which means that domains which are rotations of each other around the basepoint will have the same $h$-functions. There are other examples as well which are not rotationally equivalent, and we will return to this point later in the paper. \\ 

This paper focuses on the inverse problem. That is, given a suitable function $h$, does 
there exist a domain $D$ such that $h=h_D^z$ for some base point $z$? This problem has been studied in papers such as \cite{snipes2005realizing, barton2014new}. There are some trivial 
necessary conditions that $h$ must satisfy. To begin with, $h$ must be right continuous
and non-decreasing, with $h(r)\to1$ as $r\to\infty$. Furthermore, for some $d>0$ it must hold that $h\equiv0$ on $(0, d)$,
where $d$ is the minimum distance from $z$ to the boundary. A proposed technique to tackle
this inverse problem, found in \cite{MALWsurv}, is to approximate the $h$-function by step
functions, and consider domains of concentric circular arcs.
\begin{figure}[H]
        \includegraphics[width=120pt, height=120pt]{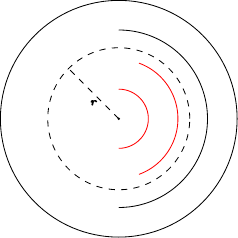}
\hspace{10mm}
        \includegraphics[width=120pt, height=120pt]{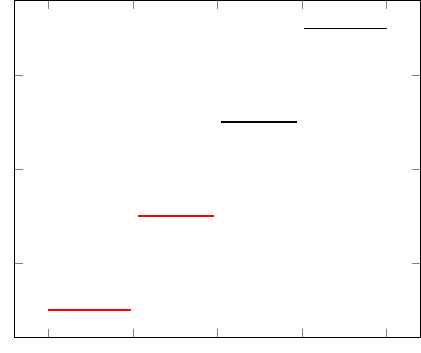}
        \caption{Step function $h$-function and the associated circle domain}\label{fig2}
\end{figure}

The method is to use these multiply-connected domains to approximate well-behaved 
simply-connected domains. In that case, under certain conditions, the limiting domain will have the
prescribed $h$-function.  \\

Our proposed solution is substantially different to the step-function program.
The methods are primarily complex-analytic and are inspired by a technique
developed by Gross to address the conformal Skorokhod embedding problem
\cite{Gross}. This problem asks for a domain $D$ such that $Re(B_{\tau_D})$ has a prescribed distribution. Gross' method was further explored in subsequent works by other authors
\cite{mariano2020conformal, boudabra2019remarks, boudabra2020new,
boudabra2024boundedness}; most notably, the authors of
\cite{mariano2020conformal} suggested that Gross' technique could be modified
to give a solution of the $h$-function inverse problem, which is essentially
the program we have followed in this paper. To see why this is a natural idea, note that $|e^z| = e^{Re(z)}$; therefore, as a rough idea, one may view the inverse problem for $h$-functions as the conformal Skorokhod embedding problem, exponentiated. \\

To make good on this idea, we first extend the notion of the $h$-function of a domain, to an $h$-function of a stopping time. This continues a general theme running through the papers \cite{GregGreen, greg, markowsky2018remark, dottedline}, with the idea being to take a classical quantity normally associated with a domain (e.g. Green's function, harmonic measure, Hardy number) and associate it instead to a stopping time of Brownian motion. This approach is discussed in detail in \cite{dottedline}, where the class $\HHH$ of stopping times was defined. We now recall this definition. 

Given a nonconstant analytic function $f$ on a domain $D$, let

\begin{align} \label{skyfall}
    \sigma_t :=  \int_0^{\tau_D \wedge t} |f'(Z_s)|^2 ds. 
\end{align}

It is not hard to see that $\sigma_t$ is strictly increasing and continuous a.s., and the same is therefore true of $C_t := \sigma_t^{-1}$. Levy's Theorem (see \cite{bass, davis, durBM, morters2010brownian}) now states that the process $\hat Z_t = f(Z_{C_t})$ is a planar Brownian motion on the random interval $[0,\sigma_{\tau_D}]$, and $\sigma_{\tau_D}$ is a new stopping time, although it is not necessarily the exit time of a domain. 

This is precisely the type of stopping time we are interested in. We let

$$
\HHH = \{\tau: \tau = \sigma_{\tau_\HH} \text{, with $\sigma$ defined by (\ref{skyfall}) for some $f$ analytic on $\HH$}\}.
$$

Note that here $\HH$ denotes the upper half-plane $\{Im(z)>0\}$; the definition given in \cite{dottedline} made use of the unit disk $\DD$ rather than $\HH$, but since $\DD$ and $\HH$ are conformally equivalent it is clear that the two definitions are equivalent. A good example of a stopping time in $\HHH$ which is not the exit time of a domain is $\inf\{t \geq 0: |arg(Z_t)| = r\}$ for $r > \pi$; this is the projection of $\tau_\HH$ under the analytic function $z \to (-iz)^{r/\pi}$.

In many cases, for $\tau \in \HHH$ the quantity 
$P_z(|Z_{\tau} - z | \leq r)$ makes sense even when $\tau$ isn't the first exit time of a domain. We
will define this quantity as the $h$-function of $\tau$ with base-point $z$. When $\tau$ is 
the first exit time from a domain $D$, then clearly $h_D^z = h_\tau^z$. Our solution to the 
inverse problem for the $h$-function will produce a stopping time in $\HHH$ rather than a domain. To be precise, our main result is the following.

\begin{theorem}\label{theorem1}
	Let $h$ be right continuous, non-negative and increasing on $[0,\infty)$ with $h(r)\to1$
	as $r\to\infty$. Denote the generalised inverse of $h$ by $g$. If $\ln(g)\in L^p([0,1])$
	for some $p>1$, then there exists a stopping time $\tau\in\mc{H}$ such that,
	$$
		h(r) = h_\tau^0(r) = P_0(|Z_\tau| \leq r).
	$$
\end{theorem}\
The paper is organised as follows: In Section \ref{sec2} we will define the relevant objects and
their relations. We will state and prove three lemmas required for the proof of Theorem
\ref{theorem1}. We will also review the subtle aspects of conformal invariance of Brownian motion.
In Section \ref{sec3} we will motivate and prove Theorem \ref{theorem1}. In Section 
\ref{sec4} we will 
present many examples where the solution to the inverse stopping time problem also solves the
original inverse problem. We will also cover the numerical approximation of the
function constructed in the proof of Theorem \ref{theorem1}. 

\section{Definitions and Preliminaries}\label{sec2}

In this section we collect some basic definitions and results we will need.

\subsection{Harmonic Measure}
The harmonic measure of a plane domain $D$ is a family of probability measures 
$\{\omega_a\}_{a\in \overline{D}}$ defined on the Borel sets of $\partial D$. It is given its name because for any 
fixed boundary set $E$ the map $a\mapsto\omega_a(E)$ is harmonic. There are three standard 
definitions of harmonic measure, and we will pass between these definitions freely and without comment, 
so for clarity we will recall them here. Harmonicity and conformal invariance can be considered either
a trivial consequence of the definition, a simple corollary to a computation, or a theorem, depending on the chosen
definition. What follows is essentially an extension of the introductory paragraph in \cite{makarov}.
\\
\begin{definition}[Complex Analytic Definition]
 Let $D\subset \C$ be simply connected, distinguish a point $a\in D$, and let $\phi_a:\D\to D$ be a 
conformal (holomorphic and bijective) map sending $0$ to $a$. Then harmonic measure 
``with pole at a'' is defined as the push-forward of normalised arc-length measure on $\D$ by 
$\phi_a$,
$$
	\omega_a(E) := \fr{1}{2\pi}(\varphi_a)_*m(E) = \fr{1}{2\pi} m(\varphi_a^{-1}(E)),
$$
for $E\in \mc{B}(\partial D)$, the Borel sets of the boundary.
\end{definition}
The harmonic
measure of the unit disk can be found using the conformal map $z\mapsto (z-a)/(1-\overline{a}z)$, and in fact the density of $\omega_a$ is exactly the Poisson kernel on the disk. Consequently, for all 
simply connected domains $D$ and fixed $E\in\mc{B}(\partial D)$ the map $a\mapsto \omega_a(E)$ 
is harmonic with boundary value $1_E$. 

This property follows trivially from the potential theory
definition of Harmonic measure, as well, which is simply the Riesz-Markov measure associated to the Poisson integral.
\begin{definition}[Potential Theoretic Definition]
    Let $D$ be a subset of $\C$ which has an non-empty, non-polar boundary (see \cite{ransford}). For a continuous and bounded $f:\partial D \to \R$, the
    solution to the Dirichlet problem on $D$ with boundary data $f$ exists and is denoted $P_D f$. For each $z\in D$ the Riesz Markov measure associated to the operator
    $f\mapsto P_Df(z)$ is the harmonic measure with pole at $z$, that is,
    $$
	    P_Df(z) = \int_{\partial D}f(w) d\omega_z(w).
    $$
\end{definition}
The measure exists as the operator $f\mapsto P_D f(z)$ is positive in the sense that $f\geq0$, implies $P_D f\geq0$. Furthermore, $\omega_z$ is a probability measure as $f\equiv1$ implies $P_D f \equiv 1$.
This definition places no requirements on the
connectivity of a domain, and agrees with the previous as 
composition with a holomorphic function preserves harmonicity. Because the boundary of 
every simply-connected strict subset of $\C$ is non-polar, this definition is an extension of the 
previous. 
This naturally leads to the probabilistic
definition of harmonic measure.  As shown initially by Kakutani \cite{Kakutani}, the unique bounded solution to the
Dirichlet 
problem can be expressed as $E_zf(Z_{\tau_D})$, where as usual $\tau_D$ represents the first 
exit from $D$, and $Z$ is Brownian motion started at $z\in D$. For the choice of $f=1_E$, combining
Kakutani's result with the potential theory definition of harmonic measure we find the following.
\begin{definition}[Probabilistic Definition]
    Let $D$ be a subset of $\C$ with $P^a(\tau_D <\infty) = 1$ for all $a\in D$ where $\tau_D$ is the first exit from $D$ from Brownian motion $Z$ started at $a\in D$. 
    Then harmonic measure with pole at $z$ is given by
        $$
	    \omega_z(E) = P_z(Z_{\tau_D}\in E),
        $$
for $E\in \mc{B}(\partial D)$.
\end{definition}
This definition is equivalent to the previous two by Levy's conformal invariance of Brownian motion,
and the strong Markov property. The complex analytic definition of harmonic measure extends to the 
cases above if we instead push-forward by the universal covering map.

\subsection{Periodic Integral Transforms}
The holomorphic function mentioned in the introduction will be constructed by an integral transform
on $\R$. For reasons that will be apparent in the next section, we must consider
transforms applied to periodic functions. Consequently, we will spend this section studying 
the form of these operators on $\R/ \Gamma$, where $\Gamma = 2\Z+1=\{2n+1:n\in\Z\}$. This will be vital when numerically
approximating the function, while also making some arguments simpler.\\

Define $T=\R / \Gamma$ and equip it with the quotient topology. Let $L^p(T)$ denote the measurable functions on $\R$ which are periodic 
with respect to $\Gamma$ and also satisfy $\int_{-1}^1 |f|^p <\infty$. These functions are in one-to-one correspondence with $L^p$ integrable
functions on $T$. If $\mc{T}$ is an integral operator defined on $L^p(\R)$ which preserves the periodicity with respect to $\Gamma$, then 
the domain of definition can be extended to $L^p(T)$ in the following manner: Let $K$ be the associated kernel, and let $f\in L^P(T)$,
$$
    \mc{T}f = \int_\R K(t,x)f(t)dt = \sum_{k=-\infty}^\infty \int _{-1}^1 K(t+2k, x) f(t) dt = \int_{-1}^1 \tilde K(t,x) f(t)dt.
$$
Here $\tilde K(x,t) = \sum_{k=-\infty}^\infty K(x+2k, t)$ is the periodic form of $K$, assuming this exists, and the final expression on the right is taken to be the definition of $\mc{T}$ on $L^p(T)$.

For example, the form obtained for the Poisson kernel (which we will require later) is contained in the following lemma.

\begin{lemma}\label{lemma:1}
The periodic Poisson integral takes the form
    \begin{align*}
        (P_y\ast f)(x) = \frac{1}{2}\int_T \fr{\sinh\pi y}{\cosh\pi y - \cos\pi(x-t)} f(t) dt.
    \end{align*}
	for $(x, y) \in T\times \R^+$. In the limit we have,
    \begin{align*}
        \lim_{y\to\infty} (P_y\ast f) = \frac{1}{2}\int_T f(t)dt.
    \end{align*}
    This exists for $f\in L^p(T)$, and furthermore the Poisson integral is bounded on $L^p(T)$ and $P_y\ast f$ is harmonic on 
    $T\times \R^+$ with $\lim_{y\to 0} (P_y\ast f) = f$ almost everywhere. 
\end{lemma}
\begin{proof}
Let $f\in L^p(T)$.
    \begin{align*}
        (P_y \ast f)(x) &= \frac{1}{\pi} \int_{\R} \frac{y}{(x-t)^2 + y^2}  f(t) dt\\
                   &= \frac{1}{\pi} \int_{-1}^1 \sum_{k\in\Z} \frac{y}{(x-t-2k)^2 + y^2}  f(t) dt.
    \end{align*}
    A well-known application of the residue theorem yields the following formula for evaluating sums of this type,
    $$
        \sum_{k=-\infty}^\infty f(k) = \sum_{z_j}\mathrm{Res}(\pi \cot(\pi z) f(z);z_j),
    $$
    where the summation on the right ranges over the poles of $g(z)=\pi \cot(\pi z) f(z)$ (see \cite{hoffman1999basic}).
    In the case of the Poisson kernel, there are two simple poles of $g$, 
    namely $z_{\pm}=\fr{1}{2}(\pm yi + t - x)$, extracting the residue at each, we find,
    $$
     \sum_{k\in\Z} \frac{y}{(x-t-2k)^2 + y^2} = i \pi 
	    \left(\cot\frac{\pi}{2}(x-t + i y) - \cot\frac{\pi}{2}(x-t - i y)\right).
    $$
    The above simplifies to the desired form using the complex definition of $\sin$ 
    and $\cos$, and this shows that $P_y\ast f$ exists. 
    The limit as $y\to\infty$ follows from the dominated convergence theorem, 
    as is boundedness through H\"{o}lder's inequality with the periodic kernel.
    Harmonicity and the limit as $y\to 0$ is carried over by the standard Poisson integral.
\end{proof}
The principle motivation for this lemma is that the Poisson integral 
in this form is much  easier to compute numerically than the standard form 
on $L_\Gamma^p(\R)$. 
An obvious difference in the Poisson integral on $L^p(\R)$ and $L^p(T)$ is the behaviour at $\infty$,
 and becomes more clear when the transform is put into the form given in the lemma.
 
 %

We also require the periodic Hilbert transform.
 
\begin{lemma}\label{lemma:2}
    The periodic Hilbert transform takes the form
    \begin{align*}
        (Hf)(x) = \frac{1}{2} \mathrm{p.v.}\int_T f(t) \cot \frac{\pi}{2} (x-t)dt.
    \end{align*}
    
    This exists for $f\in L^p(T)$, and furthermore $H$ is bounded on $L^p(T)$, for $1<p<\infty$, and commutes with the Poisson integral. 
\end{lemma}
\begin{proof}
 First we prove existence. Let $f\in L^p(T)$. Then
 
    \begin{align*}
        (H f)(x) = \frac{1}{\pi} \mathrm{p.v.} \int_\R \frac{f(t)}{x-t} dt =& \frac{1}{\pi} 
	    \lim_{\eps\to 0}\int_{|t| > \eps} \frac{f(x-t)}{t}dt \\                                         
	    =& \frac{1}{\pi} \lim_{\eps\to 0} \int_{\eps \leq |t| \leq 1} \frac{f(x-t)}{t}dt +
	\frac{1}{\pi}\int_{-1}^1 \sum_{\substack{k\in\Z\\k\neq0}}\frac{f(x-t)}{t-2k}dt 
    \end{align*}
    This follows by the periodicity of $f$. The Hilbert transform of an
    $L^p(\R)$ function exists almost everywhere. Let $f_1 :=
    \mathbf{1}_{[-2, 2]} f$, so $f_1\in L^p(\R)$ and on $(-2, 2)$ we have
    $f_1 = f$. That is, the first integral in the last equation is
    nothing but the Hilbert transform of an $L^p(\R)$ function, hence it exists
    for almost every $x$. 
    Again a routine application of the residue theorem yields the following equality.
    \begin{align*}
        \sum_{\substack{k\in\Z\\k\neq0}} \frac{1}{t-2k} = \frac{\pi}{2}
	    \cot \frac{\pi}{2}t - \frac{1}{t}
    \end{align*}
    This is bounded on $(-1, 1)$, and so the second integral term is finite for
    almost all $x$. This proves existence, the form trivially follows by
    replacing $\int_{-1}^1 f(x-t) \left(\frac{\pi}{2}\cot \frac{\pi}{2}t
    - \frac{1}{t}\right)dt$ with 
    $\lim_{\eps\to 0} \int_{\eps \leq |t| \leq 1}  f(x-t) 
	\left(\frac{\pi}{2}\cot \frac{\pi}{2}t - \frac{1}{t}\right)dt$ and simplifying.\\
    
    Now, to prove boundedness on $L^p(T)$ for $1 < p < \infty$. Let $H_\eps f :=  
	\frac{1}{\pi} \int_{|t| > \eps} \frac{ f(x-t)}{t}dt$, then
    \begin{align*}
        \Vert H_\eps  f\Vert_p &\leq \bigg\Vert \frac{1}{\pi} \int_{\eps \leq |t| \leq 1} 
	    \frac{ f(x-t)}{t}dt \bigg \Vert_p + 
        \bigg\Vert \frac{1}{\pi}\int_{-1}^1 \sum_{\substack{k\in\Z\\k\neq0}}
	    \frac{ f(x-t)}{t-2k}dt\bigg\Vert_p  \\ &\leq A\Vert f_1\Vert_p + B \Vert  
	    f\Vert_p \leq C\Vert f\Vert_p 
    \end{align*}
    The first bound follows as the Hilbert transform is bounded on $\R$ and $f_1\in L^p(\R)$ agrees 
	with $f$ on $(-1, 1)$. The second term follows by repeated application of the triangle 
	inequality. Fatou's lemma with $\eps \to 0$ gives boundedness on $L^p$ 
	for $1 < p < \infty$. \\

    Finally, as the Hilbert transform is translationally invariant it must commute with 
	convolution operators, so for $p\in(1, \infty)$ $P_y\ast Hf = H (P_y \ast f)$.
\end{proof}
    In the case where $p=\infty$, since $T$ is compact it still makes sense to
    apply the Hilbert transform to the Poisson integral. In the $\R$ case the
    Hilbert transform is bounded between $L^\infty(\R)$ and $\mathrm{BMO}$, and
    the Poisson integral of a $\mathrm{BMO}$ function remains $\mathrm{BMO}$,
    so here the iteration also is well-defined \cite{stein}. For proof of the boundedness of the Hilbert transform on $\R$ see
    \cite{GFT}. 
    
\subsection{Conformal Invariance of Brownian Motion}

If $f$ is a holomorphic function on $U\subset \C$, then an application of It\^o's lemma 
and Levy's characterisation of Brownian motion shows that $f(Z_t)$ is a time-changed planar Brownian defined on the
random interval $[0, \tau_U)$, where the time change given by equation (\ref{skyfall}). 
This result is typically referred to as ``Levy's Conformal Invariance of Brownian motion''. 
We stress, however, that injectivity is not required for this result, nor do we require the derivative of $f$ to never vanish,
meaning Brownian motion is invariant under functions which are not conformal in the two most common definitions of the word. 
We will simply call this 
``holomorphic invariance of Brownian motion''. The associated stopping time $\mc{H}$ we will denote 
$f(\tau_U)$, and will refer to $f(\tau_U)$ as the {\it projection of $\tau_U$ under $f$}.

For an example, take $D := \{z: 0 < \Re z < \ln r\}$ and $\tilde D := \{z : 1 < |z| < r\}$ with 
$f(z) = e^z$, which takes $D$ onto $\tilde D$. Here, $f$ is a (holomorphic) covering map, and it is not hard to see that 
$f(\tau_D) = \tau_{\tilde D}$ (this also follows from Proposition \ref{covering} below). 
For a more delicate example, let $S := \C\setminus(-\infty, -1]$ and $f(z) = z^2$. If we start Brownian motion 
on $[1, \infty)$, then $\tau_S$ projects under $f$ to the first time the Brownian motion hits the positive real line after it winds
about $0$ once . Therefore, in this example $f(\tau_S) \neq \tau_{f(S)}$. 

We do not want to restrict to the maps which continuously extend to the boundary, as this is far too restrictive. 
Therefore, by $f(\partial U)$ we mean the set $\{w = \lim_{z_n\to z} f(z_n):
\{z_n\}_{n\in\N}\subset U \text{ with }z_n\to z\in \partial U \}$.
For it to hold that $f(\tau_U)=\tau_{f(U)}$ we require $f(U) \cap f(\partial U) = \emptyset$; this holds for instance if $f$ is a covering
map.

\begin{proposition}\label{covering}
	If $f:U\to f(U)$ is a covering map, the projection $f(\tau_U)$ is the first exit time of Brownian motion in 
	$f(U)$.
\end{proposition}
\begin{proof}
	Suppose $w \in f(U) \cap f(\partial U)$ where $f$ is a covering map. 
	Since $w\in f(U)$ we can choose compact subset $K$ with 
	$w\in K\subset f(U)$. As $w \in f(\partial U)$, there exists some $z\in\partial U$  and a sequence 
	$\{z_n\}\subset U$ with $z_n\to z$ and $f(z_n) \to f(z)$. By taking sub-sequences if necessary, we assume
	for some large $n_0$,
	we have $f(z_n) \in K$ for all $n\geq n_0$. As $f$ is a covering map, it holds that
	$f^{-1}(K) = \bigcup_{j\in \N} K_j$, where each $K_j$ 
	is compact. It follows from the definition of a covering map that $z_n$ is 
	contained in some $K_j$ for $n\geq n_0$, which contradicts $z_n \to z$.
	Thus $f(U) \cap f(\partial U) = \emptyset$, and the statement follows.
\end{proof}
The example above with $f(z) = z^2$ on $S$ is not a covering map, and
moreover $f((-\infty, 1]) = f([1,\infty)) =[1,\infty)$. This explains the winding behaviour of the projection.

Define the cylinder $\mc{C} := \H / (2\Z+1)$. There exists a natural projection $\pi:\H \to \mc{C}$ and an obvious 
complex structure, constructed identically to that of the torus \cite{Otto}. Functions which are periodic with respect to $(2\Z+1)$
are in one to one correspondence with functions on $\mc{C}$. The Poisson integral of a function of $f\in L^p(T)$ can therefore be 
considered as a harmonic function on $\mc{C}$. Define Brownian motion
on $\mc{C}$ to be the projection of Brownian motion on $\H$ under $\pi$.
\begin{lemma}\label{lemma:3}
	Let $Z$ be Brownian motion on $\mc{C}$ started at a point $x+iy$. The distribution of $Z_{\tau_{\mc{C}}}$
	converges to a uniform distribution on $T$ as $y\to\infty$.
\end{lemma}
\begin{proof}
	The statement is clearly true, and a direct proof is simple: The density of $Z_{\tau_{\mc{C}}}$ is the Poisson kernel, and by
	Lemma \ref{lemma:1} as $y\to\infty$ the density tends to $1/2$.
\end{proof}

To put this result another way, for Brownian motion $B$ in the upper half-plane
starting at the point $yi$, the distribution of the random variable
$B_{\tau_{\HH}} (\mbox{mod } 2)$ converges to a uniform distribution as $y \to
\infty$.

\section{Proof of Theorem 1}\label{sec3}
The chain of ideas which lead to the proof are the following: For a suitable measure $\mu$ the conformal Skorokhod embedding problem
determines a domain $D$ such that $\Re Z_{\tau_D} \sim \mu$ \cite{Gross}. The inverse problem for the $h$-function is roughly the 
exponential of this, in the sense that $|\exp( Z_{\tau_\H})|=\exp(\Re(Z_{\tau_\H}))$. With Lemma \ref{lemma:3} in mind, if there exists a 
holomorphic function on $\H$ periodic with respect to $2\Z+1$, such that $\exp f(\infty) = 0$, and $\Re f = \ln h^{-1}$, then 
$\lim_{y\to\infty}|\exp(f(Z_{\tau_\H}))| = h^{-1}(U)$, where $U$ is uniform on the interval on a period of $f$ and $Z$ is started at $x+iy$. 
Thus, the projection of $\tau_\H$ under $\exp f$ is
the solution to the inverse problem for the $h$-function of a  stopping time. 
If $\tau$ corresponds to the hitting of a domain $D$, then $D$ is a solution to the original inverse problem.
\begin{proof}[Proof of Theorem 1]
	Suppose $h$ is right-continuous with $h(r)\to1$ when $r\to\infty$, let $g$ be an even extension of the generalised inverse of $h$, as 
	considered a function on $\R$ periodic with respect to $2\Z + 1$. We assume $\ln g \in L^p(T)$ for some $p>1$, so, by Lemma \ref{lemma:1},
	$u := P_y \ast \ln g$ is harmonic on $\mc{C}$, and by Lemma \ref{lemma:2} has harmonic conjugate $v:= P_y \ast (H \ln  g)$. Let $F$ be
	the corresponding holomorphic function. Suppose $k$ is a holomorphic function on $\mc{C}$, such that for all $x$, 
	$\lim_{y\to\infty}|\exp k(x+iy)| = 0$ and $|\exp k(x)| = 1$. Define the holomorphic function $f$ by,
	$$
		f(z) = \exp (F(z) + k(z)).
	$$
	This function satisfies all properties mentioned in the previous discussion. Let $Z_t$ be Brownian motion started at a point 
	$x+iy\in\mc{C}$; then $f(Z_t)$ is a time-changed planar Brownian motion started at $f(x+iy)$, stopped at $\tau= f(\tau_\C)\in\mc{H}$. 
	By Lemma \ref{lemma:3} the random variable $f(Z_{\tau_\mc{C}})$ converges to a uniform random variable on $\R/(2\Z+1)$ as $y\to\infty$.
	In the image $f(Z_t)$ converges to a planar Brownian motion started at the origin. By construction of $f$, we have:
	$$
		|\tilde Z_\tau| = \lim_{y\to\infty} |f(Z_{\tau_\mc{C}})| = \lim_{y\to\infty} \exp \ln g(Z_{\tau_{\mc{C}}}) = g(U).
	$$
	By the inverse sampling theorem, $| \tilde Z_\tau |$ is distributed as $h$: 
	$$
		P_0(|\tilde Z_\tau| \leq r) = h(r).
	$$
	That is, $h$ is the $h$-function of the stopping time $\tau$.
\end{proof}
Of course, $k(z) = i\pi z$ satisfies the conditions required, and we can also allow $k(z) = \alpha \pi i z$ by changing the definition
of $\mc{C}$ to be a quotient with a lattice which $F$ and $k$ are jointly periodic with respect to. In fact $k$ could be incommensurable
with $F$, as $|f(x)|$ will always have a period of $2$ on $\R$. The proof is identical except exchanging $\mc{C}$ with $\H$. 

It would seem that the choice of $k$ in the proof of this result would afford a
large amount of freedom, however we will show that $k(z) = \alpha \pi i z$ is
the only choice which satisfies the required conditions. In particular, we have
the following result.
\begin{proposition}
	Let $k$ be a holomorphic function on $\{z:\Re(z)<0\}$, such that $|k(iy)|=1$, $k(z+2\pi i) = k(z)$ and $\lim_{x\to-\infty} k(x+iy) = 0$. Then $k(z) =c\exp(\alpha z)$ for some integer $\alpha$ and constant $c$ with $|c|=1$. 

\end{proposition}
\begin{proof}
      Since $k$ maps the imaginary axis to
    the circle, by Schwarz reflection $k$ extends to an entire function. Define
    $f(z) =  k(\log(z))$, where we initially choose the branch of $\log$ which is holomorphic on
    $\C\setminus[0,\infty)$, with $Im(\log(z)) \in (0, 2 \pi)$. By the periodicity of $k$, the function $f$ extends to be
    continuous on $(0, \infty)$. A standard exercise in complex analysis shows
    that the path integral of $f$ must be zero over any triangle in the
    punctured plane $\C\setminus\{0\}$, and therefore by Morera's theorem $f$
    is holomorphic on $\C\setminus\{0\}$. At the origin $\log(z)$ tends to
    $-\infty$, and since $k(x+yi)$ is bounded as $x\to-\infty$, $f(z)$ is
    bounded about the origin, and hence extends to an entire function by the Casorati–Weierstrass Theorem.
    $\log(z)$ maps the circle to the imaginary axis, and since $k$ maps
    the imaginary axis to the circle, $f(z)$ maps the unit circle to itself. In
    other words, $|f(z)| = 1$ whenever $|z|=1$. The function $z \to
    1/\overline{z}$ fixes points on the unit circle, and thus the function
    $1/\overline{f(1/\overline{z})}$ is a meromorphic function which agrees
    with $f$ on the unit circle. By the identity principle, we have $f(z) =
    1/\overline{f(1/\overline{z})}$, and this implies that $f$ must be non-zero
    at all points except possibly the origin. Suppose $f$ has a zero of order
    $\alpha$ at $0$. Then $f(z)/z^\alpha$ is non-zero in $\D$. By the maximum
    and minimum
    modulus principle, the maximum and minimum must occur on the boundary. But
    $|f|$ is constant here, so they must agree and be equal to one, and thus $f(z)/z^\alpha$ is a constant $c$ with $|c|=1$. We therefore have $f(z) =
    cz^\alpha$, and we conclude $k(z) = c\exp(\alpha z)$.
\end{proof}

\begin{remark}\label{corr1}
    If the stopping time of Theorem \ref{theorem1} is the hitting time of a domain $D$, then $D$ is symmetric with respect to $\R$. This holds as 
    the only choice of $k$ is an odd function, and $u(x,0)$, $v(x,0)$ are even and odd, respectively.
\end{remark}

More generally, if we define $\partial \tau = F(\R)$, i.e. the set on which the projected Brownian motion terminates, then $\partial \tau$ is symmetric with respect to $\R$.

\section{Investigation of the Stopping Time}\label{sec4}
\subsection{Examples}
Let us start with a trivial example. Let $h$ be the $h$-function of the unit
disk with base-point $0$. Then the generalised inverse of $h$ is identically
$1$. The associated holomorphic function is $F(z) = \exp(\pi i z)$, and indeed
$F(\H) = \D$. For a more non-trivial example, recall the $h$-function for the
upper-half plane with base-point $i$. The inverse of this function is $g(x) =
\sec(\pi x/2 )$, and numerically we find 
the Hilbert transform of $\ln g$ to be
$$
    (H \ln g)(s) = \fr{1}{2} \text{p.v.} \int_{-1}^1 \ln\left(\sec\fr{\pi}{2} t\right) \cot\left(\fr{\pi}{2}(t-s)\right) dt = -\fr{\pi}{2} s.
$$
Let $F_\alpha$ be the holomorphic function on $\H$ that was constructed in the previous section with $g(x) = \sec(\pi x/2)$, and $k(z) = \alpha i\pi z$.
For the choice of $\alpha = 1$ we find the behaviour seen in Figure \ref{F1H}.

    \begin{figure}[H]
        \includegraphics[width=120pt, height=120pt]{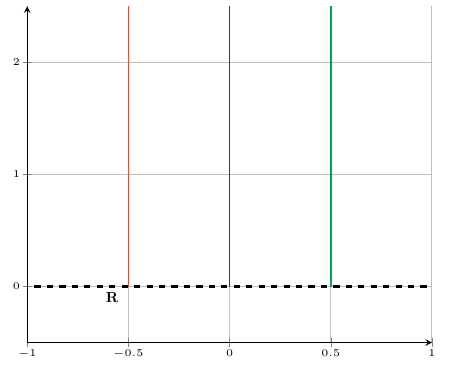}
        \includegraphics[width=120pt, height=120pt]{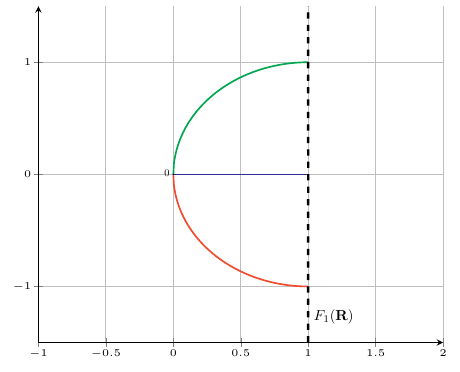}
        \caption{Domain generated from the $h$-function of $\H$ with $\alpha =1$\label{F1H}}
    \end{figure}
The left plot is $\H$ with vertical lines $x + it$ for $t\in(0,\infty)$, which
are included so we can keep track of the behaviour of $F_1$ in the image. We
see that the stopping time projects to the hitting time of a domain $D$. The
domain recovered is the simplest equivalent domain to $\H$ which is symmetric
with respect to $\R$ having $0$ being a distance of $1$ from the boundary. If
we instead choose $\alpha = 1/2$, then $F_{1/2}$ has fundamental domain $(-2,
2)$. Since $g$ has a singularity at every odd integer, increasing the
fundamental domain has the effect of adding a boundary component. If we recover
a domain from $F_{1/2}$ then it will have two boundary components. It is known
that $\C\setminus((-\infty, -1] \cup [1, \infty))$ with base point $0$ shares
the same $h$-function as $\H$ with base point $i$. This domain-base point pair
is then a good candidate, and in fact this is the behaviour we recover from
$F_{1/2}$.
    \begin{figure}[H]
        \includegraphics[width=120pt, height=120pt]{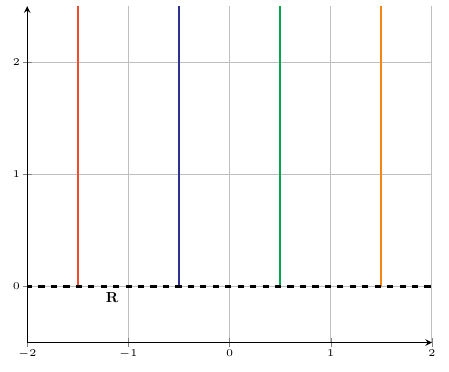}
        \includegraphics[width=120pt, height=120pt]{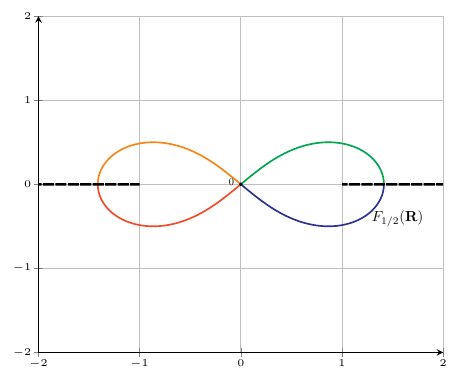}
        \caption{Domain generated from the $h$-function of $\H$ with $\alpha =1/2$}\label{F05H}
    \end{figure}
It is clear that the choice of $\alpha$ plays a role in determining the recovered domain. Suppose $F_\alpha$ $2k$ is periodic and injective on $(-k, k)\times \R^+$. 
Then $F_\alpha$ is a covering map defined on $\H$, and by Proposition \ref{covering} the associated stopping time 
is the hitting time of a domain $D$. This is the behaviour we numerically recover in the choices of $\alpha$ above.
It is often best to choose $\alpha = 1$, which minimises winding and the size of the fundamental domain. If $F_\alpha$ isn't a covering map
for this choice, then changing $\alpha$ is unlikely to remedy this. We will now investigate how $F_\alpha$ fails to be a covering map for different choices of $\alpha$.
    \begin{figure}[H]
        \includegraphics[width=120pt, height=120pt]{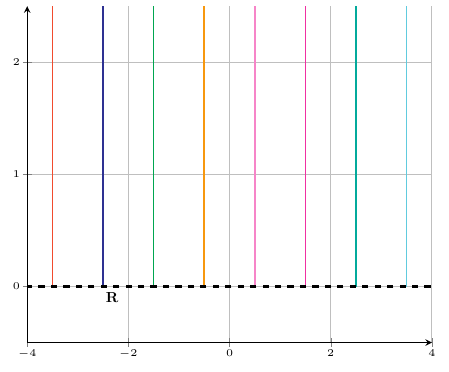}
        \includegraphics[width=120pt, height=120pt]{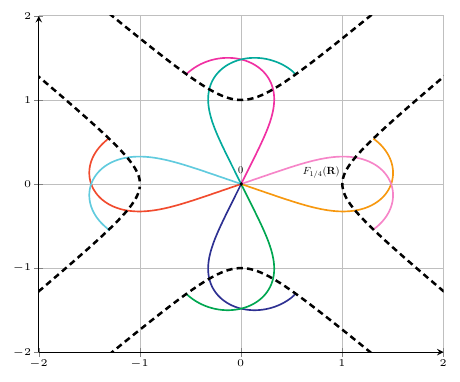}
        \caption{Domain generated from the $h$-function of $\H$ with $\alpha =1/4$}\label{F014}
    \end{figure}
Given that the image of the interior lines pass through the image of $\R$ under $F_{1/4}$, it is clear that the projection of $\tau_{\H}$
is not the hitting time of a domain $D$. While knowing the behaviour on the interior makes this clear, we can come to the same conclusion 
by the following argument. If $D$ be the component containing $0$ of $\C \setminus F_{1/4}(\R)$, then it holds that $E_0 \tau_D^p < \infty$ 
for all $p$. A theorem of Burkholder \cite{Burk} states,  
$$
  E_0 \log^+ \tau < \infty \implies c_p E_z \tau^{p/2} \leq E_z |Z_{\tau}|^p \leq C_p E_z \tau^{p/2},
$$
for some constants $c_p, C_p$ depending on $p$. It holds that $E_i \log \tau_{\H} < \infty$, and by the layer-cake representation for moments we have
$$
    E_i |Z_{\tau_{\H}}|^p = \int_0^\infty 2r^p(1-h_{\H}^i(r)) dr.
$$
Now suppose $F_{1/4}$ is a covering map from $\H$ into $D$. Then, by construction, $h_D^0(r) = h_{\H}^i(r)$. This will lead to a 
contradiction. It holds for all $p$ that $E_0[\tau_D^p] < \infty$, which implies $E_0 |Z_{\tau_D}|^p$ is finite as well. 
Then by the layer-cake representation and equality of $h$-functions, it holds that $E_i |Z_{\tau_\H}| = E_0|Z_{\tau_D}|$. 
A direct computation shows $E_i|Z_{\tau_\H}|$ is infinite, hence $F_{1/4}$ cannot be a covering map, and so the projection of $\tau_{\H}$ is 
not the hitting time of the domain $D$. For the case of $F_{1/4}$, the path in $\H$ determines which arcs
act as the boundary. From the image, for a chosen path, two collinear lines act as the boundary and the hitting of these lines will have the same exit moments as the half plane 
(see the $\alpha=1/2$ case).  
For clarity, it was mentioned in the introduction that upon choosing appropriate base points $a\in\D$ and $b\in\C\setminus\D$ 
we have $h_\D^a(r) = h_{\C\setminus\D}^b(r)$. It is well-known that $E_a[\tau_\D^p] < \infty$ for all $p$, and $E_b[\tau_{\C\setminus\D}^p] = \infty$ for all $p$.
However, the argument above does not apply as $E_b[\log^+ \tau_{\C\setminus\D}] = \infty$. 

For $F_{1/4}$ there were $4$ boundary components, in general $F_\alpha(\R)$ has $\fr{1}{2}\lcm(2, \fr{2\pi}{\alpha})$ components. If $\alpha$ is incommensurable  
with respect to $f$, then there will be infinitely many components and $F(\R)$ will be a dense subset of $r\D^c$, where $r := h^{-1}(0)$.
    \begin{figure}[H]
        \includegraphics[width=120pt, height=120pt]{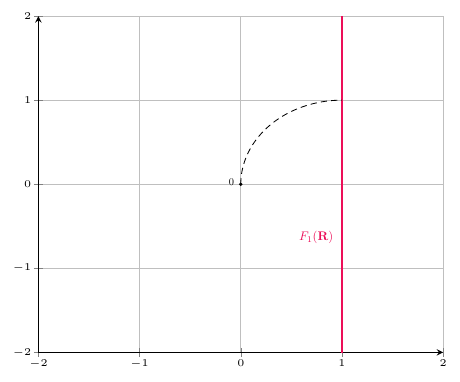}
        \includegraphics[width=120pt, height=120pt]{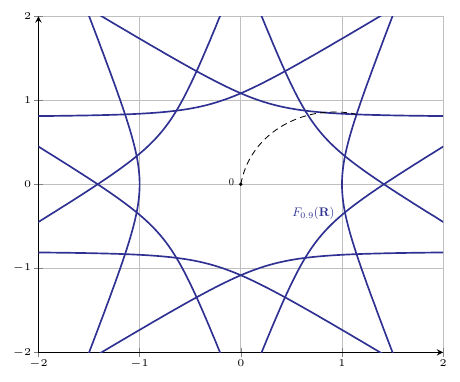}
        \includegraphics[width=120pt, height=120pt]{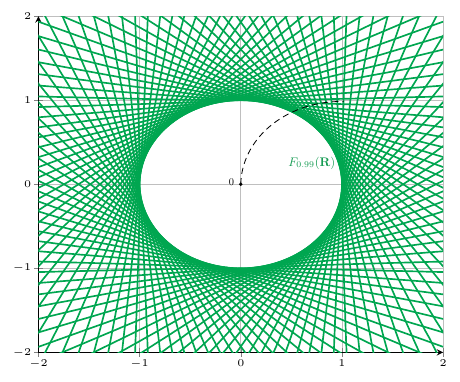}
        \caption{Domains generated from the $h$-function of $\H$ with $\alpha =1, 0.9, 0.99$}\label{F099}
    \end{figure}
The dashed line is the image of the line $\{z = 0.5+ it, \;t\in\R^+\}$, for $\alpha=0.9$ we find $10$ boundary components, and for $\alpha=0.99$ we find $100$.

We will continue exploring examples. Define $\Omega_n:= \C \setminus \{z: |z| \geq 1 \text{ and }\Arg z^n = 0\}$, so the plane with $n$ equally spaced
radial rays deleted. The $h$-function of $\Omega_n$ is given by $h_{\Omega_n}^0(r) := \fr{2}{\pi}\arctan\sqrt{r^n - 1}$, with inverse $g_n(s) = \sqrt[n]{\sec^2(\pi s/2)}$.
If $n=2$ we recover the previous example. We will project the same dashed line as the previous example, for the choices $\alpha = 1$ and $n=1, 1.5, 2, 4, 8$, we find the domains,
    \begin{figure}[H]
        \includegraphics[width=82pt, height=82pt]{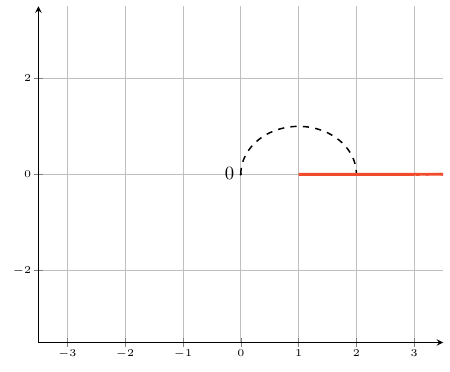}
        \includegraphics[width=82pt, height=82pt]{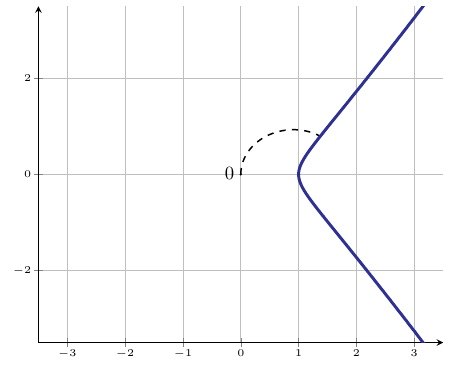}
        \includegraphics[width=82pt, height=82pt]{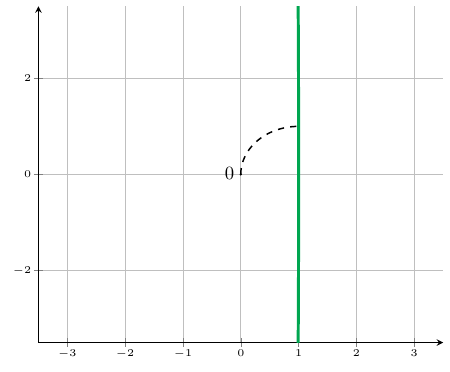}
        \includegraphics[width=82pt, height=82pt]{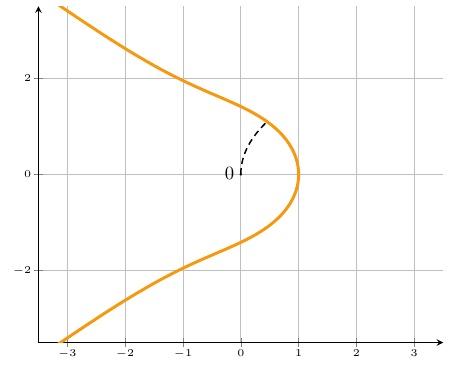}
        \includegraphics[width=82pt, height=82pt]{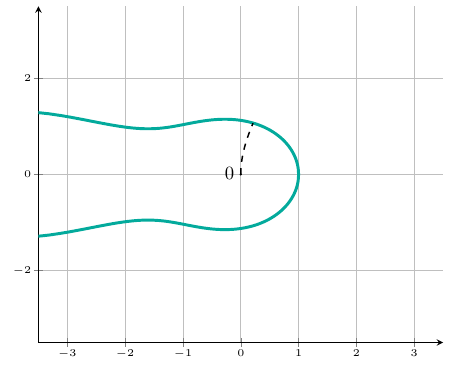}
        \caption{Domains generated with $n=1,1.5,2, 4,8$}
    \end{figure}
As $n$ becomes larger, the rays begin to fill the outside of the unit disk. Harmonic measure is supported on the bits of the boundary that ``stick out'' (see \cite{makarov}), 
these are the segments near $\partial \D$. This is intuitively seen in the behaviour of Brownian motion - if the gaps between the rays are small, Brownian motion won't
stay long between them. This is captured by the fact $g_n \to 1$ as $n\to\infty$, the inverse $h$-function of $\D$. We see in the figure above that as $n$ increases the boundary bends
down to enclose the unit disk. In the case of $n=8$, choosing $\alpha = 1/8$ yields a domain with $8$ boundary components, which yields the expected outcome of the domain recovered being $\Omega_8$ (here there is a natural choice of a domain with $8$ boundary components). For a general $n$ we can choose $\alpha = 1/n$, and the recovered domain is $\Omega_n$.

We will cover three more examples, the first emphasising results from Corollary \ref{corr1}. Given some domain $D$ which contains the origin and given the $h$-function of $D$, suppose 
the function of Theorem \ref{theorem1} produces a hitting-time of some domain $\tilde D$. Then $D$ and $\tilde D$ share the same $h$-function and only $\tilde D$ is symmetric with
respect to $\R$, a symmetrisation which leaves the $h$-function invariant. 
Consider the $h$-function of a standard equilateral triangle $T$ centred at the origin with vertices $i$, $\exp (4\pi i/3)$, $\exp(5\pi i /3)$. We recover the following two domains for
the choices of $\alpha = 1, 1/2$.
    \begin{figure}[H]
       \includegraphics[width=120pt, height=120pt]{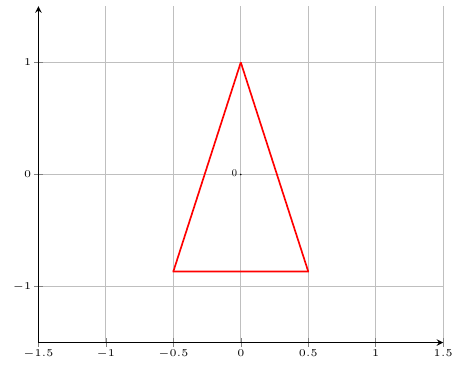}
        \includegraphics[width=120pt, height=120pt]{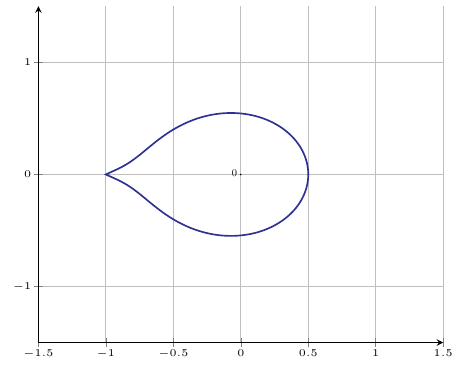}
        \includegraphics[width=120pt, height=120pt]{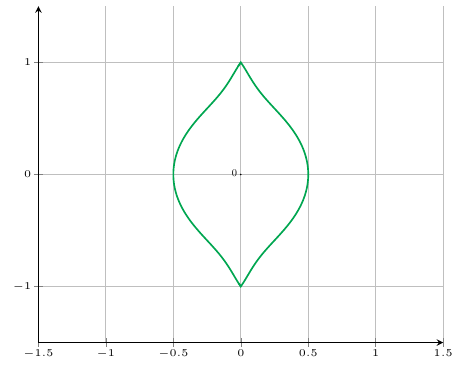}
       \caption{Domains from the $h$-function of $T$ with $\alpha=1, 1/2$.}
    \end{figure}
Of course rotating $T$ would yield a domain symmetric with respect to $\R$ with the same $h$-function. To find $h_T^0$, we used a Schwarz-Christoffel mapping, and to save time we
found the inverse with Mathematica. This was too slow computationally to determine the behaviour of the associated holomorphic function on the interior. 

The previous examples were built upon known $h$-functions. To demonstrate the robustness of the proposed method, we will investigate an example which is detached from 
the usual examples. Let $h$ be the function with inverse $g$ given by
$$
    g(s) = \fr{1}{\sqrt{1+a^{-1}}}\,\sqrt{\fr{1}{a(1-s)^n}+1}, \; \; \; s\geq 0.
$$
Now choose $\alpha = 1, n=4$ and $a=100$.
    \begin{figure}[H]
        \includegraphics[width=120pt, height=120pt]{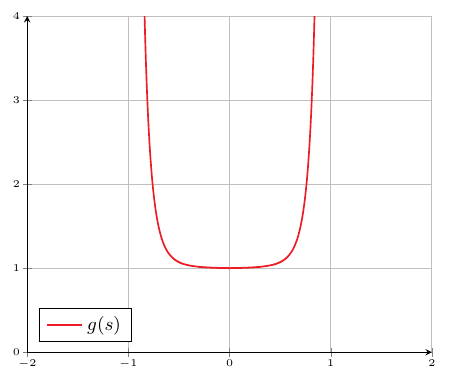}
        \includegraphics[width=120pt, height=120pt]{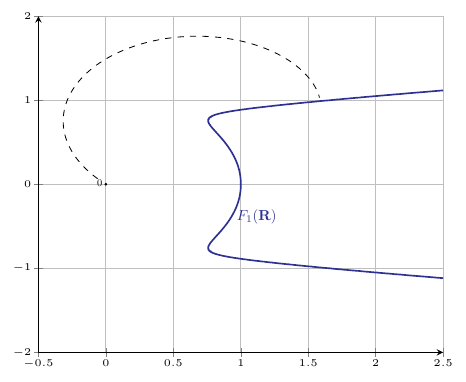}
       \caption{Domains from the $h$-function of $g$ with $\alpha=1$.}
    \end{figure}
The dashed line is the projection of the set $\{z=0.75 + it: t\in\R^+\}$. Let us call the
recovered domain $D$. The function $g(s)$ is ``close'' to $1$ on $(-0.5, 0.5)$,
therefore in the image Brownian motion has a high chance to exit at a distance
$1$ from $0$. This explains the appearance of a segment which approximates a portion of a circle in $F_1(\R)$. Away
from this interval $g$ sharply increases to $\infty$, so if Brownian motion
does not exit near $1$, there is a high probability that it exits far from $1$. In
$D$ this is obvious, if Brownian motion gets away from the origin it has ample 
space to move freely. Another way to see this is the following: the inverse
of the $h$-function of the Koebe domain $K:=\C\setminus[1,\infty)$ with
base-point $0$ is $\sec^2(\pi s/2)$. This function has a singularity of order
$2$ at $1$, and so does the function $g(s)$. Therefore we should expect the extremal behaviour of the exit
position between $D$ and $K$ to be comparable. 

In our final example we come back to the ``circle domains'' as mentioned in the introduction, domains of the type studied in \cite{MALWsurv}. Define $h:(0,\infty)\to\R$ by,
$$
    h(r) := \begin{cases} 0, \; \; &0 \leq r\leq1 \\ 0.5 \; \; &1<r\leq 2 \\ 1 \; \; &2<r \end{cases}.
$$
The generalised inverse $g$ is given by $g(s) = 1$ if $s \leq 0.5$ and $g(s) = 2$ if $0.5 < s \leq 1$. We can write down an expression for $H \ln g$ to save some accuracy and time that would 
be lost in the approximation. If we regard $g$ as an even periodic function on $\R$, $\ln g$ takes the form,
$$
    \ln g(s) = \ln(2) \sum_{k\in\Z} 1_{[\fr{1}{2}, \fr{3}{2})+2k}(s),
$$
where $[1/2, 3/2) + 2k = [1/2+2k, 3/2+2k)$. The Hilbert transform of a characteristic function is well-known, $H 1_{[a, b)} = \fr{1}{\pi} \ln| \fr{x-a}{x-b}|$, so by linearity of the Hilbert transform,
$$
    H(\ln g) = \fr{\ln 2}{\pi}\sum_{k\in \Z} \ln \left|\fr{s-2k-\fr{1}{2}}{s -2k - \fr{3}{2} }\right|.
$$
Choose $\alpha = 1$, and let $f$ be the holomorphic function associated with Theorem \ref{theorem1}. 
\begin{figure}[H]
        \includegraphics[width=130pt, height=130pt]{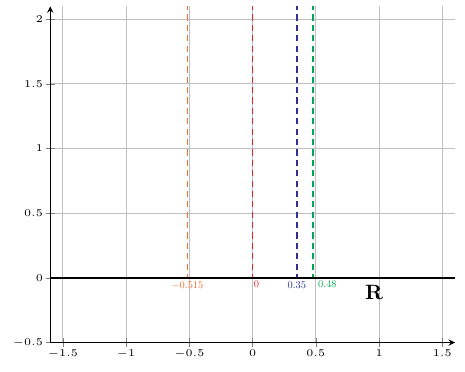}
        \includegraphics[width=130pt, height=130pt]{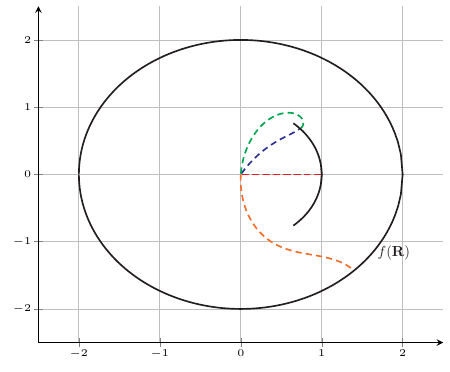}
        \caption{Domain recovered from the $h$-function of a circle domain.}\label{fig8}
\end{figure}
Let the inner circular arc be denoted by $A$. By construction of $h$, if Brownian motion in $\mc{C}$ exits on $(0.5, 1)$ then in the image it exits through the outer circle.
From the figure we see that $f$ is $2$ to $1$ on $f^{-1}(A)$, which distinguishes which side of $A$ Brownian motion exits through. Approximately, we find that when the exit
position is between $(-0.5, -0.4)\cup(0.4, 0.5)$ in the image, Brownian motion exits on the right side of $A$, while if the exit position is within $(-0.4, 0.4)$ 
in the image it exits on the left side of $A$. From this we see there is approximately a probability of $10\%$ that Brownian motion exits on the right side. 
This example shows that the proposed method can recover non-simply connected domains. \\

The relation between the stopping time and the boundary of the image can be well understood by the theory of Riemann surfaces. This is not our central focus, but we will briefly explain. Let $f$ be the function of Theorem \ref{theorem1}, and consider the Riemann surface $X$ on which $f^{-1}$ is single valued (a construction can be found in \cite{ahlfors}). 
The examples of Figure \ref{F014}, Figure \ref{F099} and the definition of $\mc{H}$ hint that the projected stopping time behaviour can be understood as the projection
of the explosion time of Brownian motion defined on $X$. Typically, one pictures $X$ as being spread out in twisted sheets over the plane. The boundary of one sheet can lie over the interior of another, and this explains the complicated behaviour of the projected hitting times.
In the previous examples the projected processes pass through the image of boundary or seem to stop in the interior of the domain, and from the perspective of $X$ this occurs simply due to the Brownian motion travelling on different sheets of the Riemann surface. If $f$ is a covering map, the boundaries of the sheets lie directly over each other, and there is no difficulty in the projected stopping time. More details on this idea will be contained in the second author's forthcoming PhD thesis.

\subsection{Numerical Construction of the Map }
Numerically our only challenges are the accurate computation of Poisson integral and the Hilbert transform.
We employed the ``tanh-sinh'' quadrature method, discovered by H. Takahasi and M. Mori in 1974. The error can be expressed as $O(\exp(-CN/\ln N))$, where $N$ is the number of
function evaluations\cite{quad}. First define $\phi(t) = \tanh(\pi/2 \sinh(t))$. Then by substitution and the trapezoidal rule with spacing $h$ we obtain,
\begin{align*}
    \int_{-1}^1 f(t) dt =& \int_{-\infty}^{\infty} f(\phi(t)) \phi'(t) dt \\&\approx 
    \fr{h\pi}{2}\sum_{k=-\infty}^\infty \fr{\cosh(kh)}{\cosh^2(\fr{\pi}{2}\sinh kh)} f\left(\tanh(\fr{1}{2}\pi\sinh kh)\right).
\end{align*}
As a choice, if we calculate this summation to $M$ terms, we choose $h = 2/M$. 
The integrand decays at a double exponential rate, so this method is suited for integrating functions that have singularities at the end points. We often are in this case as if a domain is unbounded, 
then the inverse of it's $h$-function has a singularity at $1$. 
With this in mind, it was crucial that we found periodic expressions for the Poisson integral and Hilbert transform. We will first explain optimisations we used for the Poisson integral. 

When $y\to0$ 
the Poisson kernel tends to a delta spike at $x$. We are then forced to make $h$ very small. However most of this computation is wasted, as the mass of the integral
is at $x$. We instead compute
$$
         (P_y\ast f)(x) =  \int_{-1}^x\fr{\sinh\pi y}{\cosh\pi y - \cos\pi(x-t)} f(t) dt 
         + \int_{x}^1\fr{\sinh\pi y}{\cosh\pi y - \cos\pi(x-t)} f(t) dt,
$$
and then make a linear substitution to each integral so the bounds are $-1 \leq
t\leq 1$. The chosen quadrature method sends the end points to $\infty$. 
Given the rapid decay of the integrand decays, the delta spike behaviour can then
be accurately computed. If we restrict the sets in $\H$ we wish to
compute to $\{x + it:  t\in\R^+\}$, we pre-compute every term in the
quadrature formula which only involve $x$. This is necessary to numerically study the interior behaviour of the constructed holomorphic function.

For the Hilbert transform, one could use the fast Fourier transform algorithm,
as $\mc{F}(H(f))(\zeta) = -i\sgn(\zeta)\mc{F}f(\zeta)$. However, given the bounded domain we opted for the quadrature method instead. In computing the integral
$$
    (Hf)(x) = \lim_{\eps\to0} \int_{\eps}^1 (f(x-t) - f(x+t))\cot\fr{\pi}{2} t dt,
$$
again we make a linear substitution so the bounds are $-1\leq t\leq 1$, and
we choose some sufficiently small $\eps$. We pre-compute any terms which do not
involve $t$, and then reuse these computations
wherever necessary.

When the Hilbert transform of a function is not known, the largest obstacle in
computing the interior behaviour of the holomorphic function of Theorem
\ref{theorem1} is the $P_y\ast(H\ln g)$ term. If we even want to compute one
value of $P_y \ast H f$, we will need to compute the Hilbert transform for
every-point required in the quadrature expression for the Poisson integral. To
cut down on computation, we pre-compute the Hilbert transform on each of the points 
which appear in the computation for the Poisson integral.

\bibliographystyle{plainurl}
\bibliography{citation}
\end{document}